\documentclass[final]{siamart0516}
\usepackage{lipsum}
\usepackage{amsfonts}
\usepackage{graphicx}
\usepackage{epstopdf}
\usepackage{algorithmic}

\usepackage{graphicx,epstopdf} 
\usepackage[caption=false]{subfig} 
\ifpdf
  \DeclareGraphicsExtensions{.eps,.pdf,.png,.jpg}
\else
  \DeclareGraphicsExtensions{.eps}
  \usepackage{amsopn}

\fi

\headers{Hamiltonian with a potential of the fourth degree}{Acosta-Hum\'anez, Alvarez-Ram\'{\i}rez, Stuchi}

\title{Non-integrability of  the   Armbruster-Guckenheimer-Kim quartic Hamiltonian  through Morales-Ramis theory}

\author{P. Acosta-Hum\'anez, 
\thanks{School of Basic and Biomedical Sciences, Universidad Sim\'on Bol\'{\i}var,
Barranquilla - Colombia (\email{primitivo.acosta@unisimonbolivar.edu.co;\,\, primi@intelectual.co})}
  \and
M. Alvarez-Ram\'{\i}rez, \thanks{Departamento de Matem\'aticas, UAM-Iztapalapa, San Rafael Atlixco 186, Col. Vicentina, 09340 Iztapalapa, M\'exico city, M\'exico. (\email{mar@xanum.uam.mx})}
  \and
T. J. Stuchi, \thanks{Instituto de F\'{\i}sica - Universidade Federal do Rio de Janeiro, Rio de Janeiro, RJ CP: 68528 - CEP: 21941-972, Brazil.
(\email{tstuchi@if.ufrj.br})}}

\begin{document}

\maketitle

\begin{abstract}
We show the non-integrability of the three-parameter   Armburster-Guckenheimer-Kim quartic Hamiltonian using 
Morales-Ramis theory, with the exception of the three already known integrable cases.  We use
Poincar\'e sections to illustrate the breakdown of regular motion for some parameter values. 
\end{abstract}

\begin{keywords}
Hamiltonian, integrability of dynamical systems, differential Galois theory, Legendre equation,  Schr\"odinger equation.
\end{keywords}

\begin{AMS}
37J30, 70H06, 12H05
\end{AMS}

\section{Introduction}
Mechanical and physical systems whose energy is conserved are  usually modelled by Hamiltonian systems with $n$ degrees of freedom.  
These systems have  a Hamiltonian function $H$ as  a first integral and we say that the Hamiltonian system is integrable in the Liouville sense (often referred to as complete integrability) if there exist $n-1$ 
additional smooth first integrals in involution  $\{I_j,H\}=0$, $j=1,\dots ,n-1$, which are functionally independent, see \cite{am}.

There are many physics and mathematical papers  with examples of numerical solutions of Hamiltonian systems with two or more degrees freedom 
 giving evidence of non-integrability. In the present paper,  using  Morales-Ramis theory \cite{morales} and some new contributions to this ,  we examine rigorously 
 the non-integrability  of the Hamiltonian of two degrees of freedom with Armbruster-Guckenheimer-Kim potential given by
\begin{equation}\label{ham-agk}
H(x,y)= \frac{1}{2}(p_x^2+p_y^2) -\frac{\mu}{2}(x^2+y^2)-\frac{a}{4} (x^2+y^2)^2- \frac{b}{2} x^2y^2.
\end{equation}
where $\mu$, $a$ and  $b$ are real parameters. 

In 1989 Simonelli and Gollub \cite{sg}  carried out experimental work with
surface waves in a square and rectangular containers subject to vertical oscillations. Their methods allowed measurements of both stable and unstable fixed points (sinks, sources, and saddles), and the nature of the bifurcation sequences  clearly established and reported. 
 Later,  motived by these experiments reported in \cite{sg},  Armbruster et. al. \cite{agk} derived a dynamical model starting with a normal form given by 
 four first order differential equations depending on several parameters. Their system provides 
a general description of the codimension two bifurcation problem.
 To restrict the model they performed a  rescaling of the variables and parameters.  
When one of the scaling parameters $\varepsilon = 0$, the system of  equations becomes Hamiltonian with
the Hamiltonian function giving by (\ref{ham-agk}) which has $D_4$-symmetry.
Armbruster et. al \cite{agk}  found  a {\em large} parameter region of chaotic behaviour.
They argue that the existence of chaos in a given region suggests that  the Hamiltonian (\ref{ham-agk}) is  non-integrable.

As early as 1966  a  similar potential  was proposed  by Andrle \cite{andrle} 
as a dynamical model for a stellar system with an axis and a plane of symmetry given by 
$$V=\frac{1}{2} (C^2{\bar\omega}^{-2}-A^2\xi^2-B^2z^2+\frac{1}{2}a\xi^4+\frac{1}{2}\gamma z^4+2\beta\xi^2z^2)$$
where $\xi = {\bar\omega}- {\bar\omega}_0$, $A$, $B$, $\alpha,\gamma$ are positive constant, $\beta$ is a constant and
$C$ is the value of the angular momentum integral.
This potential  has often been used in the study of galactic dynamics, see  among others \cite{contra}.

In particular,  if  $\mu=-1$  in  (\ref{ham-agk}) we have the Hamiltonian studied  by  Llibre and Roberto \cite{llibre}. These authors examined  its ${\mathcal C}^1$ non-integrability in the sense of
Liouville-Arnold, see \cite{am}. More recently,  Elmandouh \cite{elmandouh} studied the dynamics of 
rotation of a nearly axisymmetric galaxy rotating 
with a constant angular velocity  around a fixed axis
using (\ref{ham-agk})  in a rotating reference frame. He also examined the  non-integrability by means of
 Painlev\'e analysis.
 
The goal of this present paper is to show  that the Hamiltonian (\ref{ham-agk}) is, in fact, non-integrable except for the cases $b=0$, $a=-b$ and $b=2a$ mentioned above.
The potential of the Hamiltonian (\ref{ham-agk}) can be written in the form of the sum of two  homogeneous components 
$$
V(x,y)= -\frac{\mu}{2}(x^2+y^2)-\frac{a}{4} (x^2+y^2)^2- \frac{b}{2} x^2y^2 =V_{min} + V_{max},
$$
where $V_{min}$ and  $V_{max}$ are quadratic and quartic polynomial, respectively.

Maciejewski and Przybylska \cite{maria}  and Bostan et. al. \cite{combot}  have made advances on the Morales-Ramis Theorem for homogeneous potentials. In particular \cite{combot} has found a way of applying Morales-Ramis Theorem for homogeneous potentials in a very efficient way.  
For all Darboux points, Bostan et. al. \cite{combot} implemented an algorithm that allows to determinate 
necessary integrability conditions  of homogeneous  potentials up to degree 9.
We use this  algorithm to obtain necessary integrability conditions for the quartic potential.
Then using an old result of Hietarinta \cite{hietarinta} and  Yoshida \cite{yoshida} for polynomial potentials   the necessary integrability conditions  for AGK potential is obtained.
 The above necessary conditions imply  integrability by rational functions. The general Morales-Ramis 
 Theorem (see \cite{morales}) complements the study by given necessary conditions for rational or meromorphic integrability, depending on the kind of singularity of the normal variational  equation.
 
The structure of the present paper is as follows. In Section \ref{sec1}, we briefly review some preliminary
notions and results of Morales-Ramis theory to be used later. The main results will be given in
Section \ref{resultados}.  
In Section \ref{sec-poincare} we use Poincar\'e sections to exhibit the breakdown in regular motion suggesting non-integrability.

\section{Preliminars}\label{sec1}
In this section we describe  some basic facts concerning the Morales-Ramis theory. 

\subsection{Differential Galois Theory}
Picard-Vessiot theory is the Galois theory of linear differential
equations, also known as differential Galois theory. We present here some of its main definitions and results, and  refer the reader to~\cite{vasi} for a wide theoretical background.

We start by recalling some basic notions on algebraic groups and, afterwards, Picard-Vessiot theory is introduced.
    
An algebraic group of matrices $2\times 2$ is a subgroup $G\subset\mathrm{GL}(2,\mathbb{C})$
defined by means of algebraic equations in its matrix elements and in the inverse of its determinant. 
    It is well known that any algebraic group $G$ has a unique connected normal
    algebraic subgroup $G^0$ of finite index. In particular, the
    \emph{identity connected component $G^0$} of $G$ is defined as the largest connected
    algebraic subgroup of $G$ containing the identity. In the case that $G=G^0$ we say that $G$ is a
    \emph{connected group}. Moreover, if $G^0$ is abelian we say that $G$ is \emph{virtually abelian}.

Now, we briefly introduce Picard-Vessiot Theory.
    
    First, we say that $\left( \mathbb{K}, \phantom{i}' \ \right)$ - or, simply,
    $\mathbb{K}$ - is a \emph{differential field} if $\mathbb{K}$ is a commutative field
    of characteristic zero, depending on $x$ and $\phantom{i}'= d/dx$ is a
    derivation on $\mathbb{K}$ (that is, satisfying $(a+b)'=a'+b'$
    and $(a\cdot b)'=a'\cdot b+a \cdot b'$ for all $a,b \in \mathbb{K}$). We
    denote by $\mathcal{C}$ the \emph{field of constants of $\mathbb{K}$},
    defined as $\mathcal{C}=\left\{ c \in \mathbb{K} \ | \ c'=0
    \right\}$.
    
    We will deal with second order linear homogeneous differential
    equations, that is, equations of the form
    \begin{equation}
    \label{soldeq}
    y''+b_1y'+b_0y=0,\quad b_1,b_0\in \mathbb{K},
    \end{equation}
    and we will be concerned with the algebraic structure of their solutions. Moreover, along this work, we will refer
    to the current differential field as the smallest one containing the field of coefficients of this differential equation.
    
    Let us suppose that $y_1, y_2$ is a basis of solutions of equation
    \eqref{soldeq}, i.e., $y_1, y_2$ are linearly independent over $\mathcal{C}$
    and every solution is a linear combination over $\mathcal{C}$ of
    these two solutions. Let $\mathbb{L}= \mathbb{K}\langle y_1, y_2 \rangle =\mathbb{K}(y_1, y_2, y_1',
    y'_2)$ be the differential extension of $\mathbb{K}$ such that $\mathcal{C}$
    is the field of constants for $\mathbb{K}$ and $\mathbb{L}$. In this terms, we say
    that $\mathbb{L}$, the smallest differential field containing $\mathbb{K}$ and
    $\{y_{1},y_{2}\}$, is the \textit{Picard-Vessiot extension} of $\mathbb{K}$
    for the differential equation~\eqref{soldeq}.
    
    The group of all the differential automorphisms of $\mathbb{L}$ over $\mathbb{K}$
    that commute with the derivation $\phantom{i}'$  is called the
    \emph{differential Galois group} of $\mathbb{L}$ over $\mathbb{K}$ and is denoted by ${\rm
         DGal}(\mathbb{L}/\mathbb{K})$. This means, in particular, that for any $\sigma\in
    \mathrm{DGal}(\mathbb{L}/\mathbb{K})$, $\sigma(a')=(\sigma(a))'$ for all $a\in \mathbb{L}$
    and that $\sigma(a)=a$ for all $a\in \mathbb{K}$. Thus, if $\{y_1,y_2\}$ is
    a fundamental system of solutions of~\eqref{soldeq} and $\sigma \in
    \mathrm{DGal}(\mathbb{L}/\mathbb{K})$ then $\{\sigma y_1, \sigma y_2\}$ is also a
    fundamental system. This implies the existence of a non-singular
    constant matrix
    \[
    A_\sigma=
    \begin{pmatrix}
    a & b\\
    c & d
    \end{pmatrix}
    \in \mathrm{GL}(2,\mathbb{C}),
    \]
    such that
    \[
    \sigma
    \begin{pmatrix}
    y_{1}&
    y_{2}
    \end{pmatrix}
    =
    \begin{pmatrix}
    \sigma (y_{1})&
    \sigma (y_{2})
    \end{pmatrix}
    =\begin{pmatrix} y_{1}& y_{2}
    \end{pmatrix}A_\sigma.
    \]
    This fact can be extended in a natural way to a system
    \[
    \sigma
    \begin{pmatrix}
    y_{1}&y_2\\
    y'_1&y'_{2}
    \end{pmatrix}
    =
    \begin{pmatrix}
    \sigma (y_{1})&\sigma (y_2)\\
    \sigma (y'_1)&\sigma (y'_{2})
    \end{pmatrix}
    =\begin{pmatrix} y_{1}& y_{2}\\y'_1&y'_2
    \end{pmatrix}A_\sigma,
    \]
    which leads to a faithful representation $\mathrm{DGal}(\mathbb{L}/\mathbb{K})\to
    \mathrm{GL}(2,\mathbb{C})$ and makes possible to consider
    $\mathrm{DGal}(\mathbb{L}/\mathbb{K})$ isomorphic to a subgroup of $\mathrm{GL}(2,\mathbb{C})$
    depending (up to conjugacy) on the choice of the fundamental system $\{y_1,y_2\}$.
    
    One of the fundamental results of the Picard-Vessiot theory is the
    following theorem (see~\cite{ka,kol}).
    
    \begin{theorem}  The Galois group $\mathrm{DGal}(\mathbb{L}/\mathbb{K})$ is an
         algebraic subgroup of $\mathrm{GL}(2,\mathbb{C})$.
    \end{theorem}
    
We say that Eq. \eqref{soldeq} is \emph{integrable} if the Picard-Vessiot extension
    $\mathbb{L}\supset \mathbb{K}$ is obtained as a tower of differential fields
    $\mathbb{K}=\mathbb{L}_0\subset \mathbb{L}_1\subset\cdots\subset \mathbb{L}_m=\mathbb{L}$ such that
    $\mathbb{L}_i=\mathbb{L}_{i-1}(\eta)$ for $i=1,\ldots,m$, where either
    \begin{itemize}
         \item[$(i)$] $\eta$ is {\emph{algebraic}} over $\mathbb{L}_{i-1}$; that is, $\eta$ satisfies a
         polynomial equation with coefficients in $\mathbb{L}_{i-1}$.
         \item[$(ii)$] $\eta$ is {\emph{primitive}} over $\mathbb{L}_{i-1}$; that is, $\eta' \in \mathbb{L}_{i-1}$.
         \item[$(iii)$] $\eta$ is {\emph{exponential}} over $\mathbb{L}_{i-1}$; that is, $\eta' /\eta \in \mathbb{L}_{i-1}$.
    \end{itemize}

If $\eta$     is obtained as a  combination of (i), (ii) and (iii), we say that $\eta$ is \emph{ Liouvillan}.
    Usually in terms of  Differential Algebra's terminology we say that
    Eq. \eqref{soldeq} is integrable if the corresponding
    Picard-Vessiot extension is Liouvillian. Moreover, the
    following theorem  holds.
    
    \begin{theorem}[Kolchin]
         \label{LK}
         Equation~\eqref{soldeq} is integrable if
         and only if $\mathrm{DGal}(\mathbb{L}/\mathbb{K})$ is virtually solvable, that is, its identity component
         $(\mathrm{DGal}(\mathbb{L}/\mathbb{K}))^0$ is solvable.
    \end{theorem}

Let us notice that for a second-order linear differential
equation  $y'' =r(x)y$
with coefficients in $\mathbb{K}$, 
the only connected non-solvable group is $SL(2,\mathcal{C})$, see \cite{ka,morales,vasi} .

The following result can be found in \cite{almp}.
\begin{theorem}[Acosta-Hum\'anez, Morales-Ruiz, L\'azaro \& Pantazi, \cite{almp}]\label{almp}
         Legendre differential equation
         \begin{equation}\label{eleg}
         (1-x^2){d^2y\over dx^2} -2x{dy\over dx}+\left({\nu(\nu+1)}-{\tilde\mu^2\over 1-x^2}\right)y=0
         \end{equation} is integrable if and only if,
         either
         \begin{enumerate}
              \item $\tilde\mu\pm \nu\in\mathbb{Z}$ or $\nu\in\mathbb{Z}$, or
              \item $\pm \tilde\mu$, $\pm (2\nu+1)$ belong to one of the following seven families
              \[
              \begin{array}{|c|c|c|c|} \hline
              \mbox{Case} &\tilde\mu \in & \nu \in & \tilde\mu + \nu \in \\ \hline
              (a) & \mathbb{Z} +\frac{1}{2}  & \mathbb{C}                          &                 \\[1.2ex]  \hline
              (b) & \mathbb{Z}\pm\frac{1}{3} & \frac{1}{2}\mathbb{Z}\pm\frac{1}{3} & \mathbb{Z}+\frac{1}{6} \\[1.2ex]  \hline
              (c) & \mathbb{Z}\pm\frac{2}{5} & \frac{1}{2}\mathbb{Z}\pm\frac{1}{5} & \mathbb{Z}+\frac{1}{0} \\[1.2ex]  \hline
              (d) & \mathbb{Z}\pm\frac{1}{3} & \frac{1}{2}\mathbb{Z}\pm\frac{2}{5} & \mathbb{Z}+\frac{1}{10} \\[1.2ex]  \hline
              (e) & \mathbb{Z}\pm\frac{1}{5} & \frac{1}{2}\mathbb{Z}\pm\frac{2}{5} & \mathbb{Z}+\frac{1}{10} \\[1.2ex] \hline
              (f) & \mathbb{Z}\pm\frac{2}{5} & \frac{1}{2}\mathbb{Z}\pm\frac{1}{3} & \mathbb{Z}+\frac{1}{6} \\[1.2ex] \hline
              \end{array}       \]
\end{enumerate}
    \end{theorem}
    We recall that the singularities of Legendre equation \eqref{eleg} are of the regular type.
    \begin{theorem}[Acosta-Hum\'anez \& Bl\'azquez-Sanz, \cite{acbl}]\label{acbl}
         Assume $\mathbb{K}=\mathbb{C}(e^{i t})$. The differential Galois group of extended Mathieu differential equation 
$$
 \ddot y=(a+b\sin(t)+c\cos(t))y,\quad |b| + |c|\neq 0$$ 
         is isomorphic to $\mathrm{SL(2,\mathbb{C})}$.  
    \end{theorem}     
 Let us remark that the  variable change  $x =e^{it}$ makes  the differential
equation algebraic, such that $x=\infty$ is an irregular singularity.
 
\subsection{Morales Ramis Theory}
A Hamiltonian $H$ in $\mathbb{C}^{2n}$ is called \emph{integrable in the sense of Liouville} if
there exist $n$ independent first integrals of the Hamiltonian system in involution. We 
say that $H$ is integrable \emph{in terms of rational functions} if we can
find a complete set of integrals within the family of rational functions.
Respectively, we can say that $H$ is integrable \emph{in terms of
    meromorphic functions} if we can find a complete set of integrals within the
family of meromorphic functions.
Morales-Ramis theory relates integrability of Hamiltonian systems in the Liouville sense with the integrability of Picard-Vessiot theory in terms of differential Galois theory, see \cite{morales, morasurvey, mora1, mora2}.

Let $V$ be a meromorphic homogeneous potential.
We say that 
$c\in \mathbb{C}^n\setminus\{0\}$ is a {\em Darboux point} if there exists $\alpha\in \mathbb{C}\setminus\{0\}$
such that
$$
\frac{\partial}{\partial q_j} V(c) = \alpha c_j   \qquad \forall \; j=1,\dots , n. 
$$

\begin{theorem}[Morales-Ramis, \cite{morales, mora1,mora2}]\label{mora}
    \label{:MR} Let $H$ be a Hamiltonian in $\mathbb{C}^{2n}$, and $\gamma$ a
    particular solution such that the variational  equation associated to the system has regular (resp. irregular)
    singularities at the points of $\gamma$ at infinity. If $H$ is
    integrable by terms of meromorphic (rational) functions,
    then the differential Galois group of the variational equation is virtually abelian.
\end{theorem}

\begin{theorem} [Morales-Ramis \cite{morales, mora1,mora2}] Let $V \in \mathbb{C}(q_1,q_2)$ be a homogeneous rational function of homogeneity degree $k\ne -2,0,2$, and let $c \in \mathbb{C}^2\in\ 0\}$  be a Darboux point of $V$. If the potential $V$ is integrable, then for any eigenvalue $\lambda$ of the Hessian matrix of $V$ at $c$, the pair $(k, \lambda )$ belongs to the Table \ref{table:MR}, 
for some $j \in \mathbb{Z}$. 
\begin{table}[ht]
\centering 
\begin{tabular}{|c|c||r|c |} 
\hline
$k$ & $\lambda$ & $k$ & $\lambda$ \\
\hline
$\mathbb{Z}^*$ & $\frac{1}{2}jk (jk + k -2)$ & $-3$ &  $-\frac{25}{8} + \frac{1}{8} (\frac{12}{5}+ 6j)^2$\\
\hline 
$\mathbb{Z}^*$ & $\frac{1}{2}(jk+1) (jk + k -1)$ & $3$ &  $-\frac{1}{8} + \frac{1}{8} (2+ 6j)^2$\\
\hline
$-5$ & $-\frac{49}{8} + \frac{1}{8}(\frac{10}{3} + 10j)^2$ & $3$ &  $-\frac{1}{8} + \frac{1}{8} (\frac{3}{2}+ 6j)^2$\\
\hline
$-5$ & $-\frac{49}{8} + \frac{1}{8}(4 + 10j)^2$ & $3$ &  $-\frac{1}{8} + \frac{1}{8} (\frac{6}{5}+ 6j)^2$\\
\hline
$-4$ & $-\frac{9}{2} + \frac{1}{2}(\frac{4}{3} + 4j)^2$ & $3$ &  $-\frac{1}{8} + \frac{1}{8} (\frac{12}{5}+ 6j)^2$\\
\hline
$-3$ & $-\frac{25}{8} + \frac{1}{8}(2 + 6j)^2$ & $4$ &  $-\frac{1}{2} + \frac{1}{2} (\frac{4}{3}+ 4j)^2$\\
\hline
$-3$ & $-\frac{25}{8} + \frac{1}{8}(\frac{3}{2} + 6j)^2$ & $5$ &  $-\frac{9}{8} + \frac{1}{8} (\frac{10}{3}+ 10j)^2$\\
\hline
$-3$ & $-\frac{25}{8} + \frac{1}{8}(\frac{6}{5} + 6j)^2$ & $5$ &  $-\frac{9}{8} + \frac{1}{8} (4+ 6j)^2$\\
\hline
\end{tabular}
\caption{Morales-Ramis table.} 
\label{table:MR} \end{table}
\end{theorem}

\subsection{Galoisian integrability of Sch\"odinger equation}\label{sec:sch}
The analysis of the  integrability of the Schr\"odinger equation using differential Galois theory has been studied in \cite{acbook,acmowe}, using as main tools \emph{Kovacic algorithm} and \emph{Hamiltonian algebrization}; wave functions and spectrum of Schr\"odinger equations with potentials known as shape invariant were obtained.

It is well known that the one-dimensional stationary non-relativistic Schr\"odinger equation is giving by
\begin{equation}\label{eq-schr1}
\Psi'' = (V(y)-\lambda)\Psi,
\end{equation}
where $V$ is the potential and  $\lambda$ is called energy.

A special class of potentials are the called P\"oschl-Teller  (see \cite{poschl,T}) which transforms 
(\ref{eq-schr1}) into
\begin{equation}\label{pteq}
\Psi''=\left(-{\ell(\ell+1)\over \cosh^2 y}+n^2\right)\Psi, \qquad y\in \mathbb{R},
\end{equation}
where  $n,\ell\in\mathbb{Z}^+$. 
It is possible to algebrize this equation by means of the change  of variables  $z = \tanh y$. This substitution transforms the Schr\"odinger equation (\ref{pteq})  into
\begin{equation}\label{eq-legendre}
 (1-z^2){d^2\Psi \over dz^2} -2z{d\Psi\over dz}+\left({\ell(\ell+1)}-{n^2\over 1-z^2}\right) \Psi=0,
 \end{equation}
 which is a  differential equation for the associated Legendre function. Two linearly independent solutions of
(\ref{eq-legendre})  are given by the associated Legendre functions of the first and second kind, respectively.

\section{Main results}\label{resultados}
We start by giving the  known integrable cases of the Hamiltonian (\ref{ham-agk})
with associated  differential equations 
\begin{eqnarray}\label{eq-ham}
\dot{x} &=& p_x, \nonumber\\
\dot{y} &=& p_y,\nonumber\\
p_x &=& \mu x+ a(x^2+y^2)x+bxy^2, \nonumber\\
p_y &=& \mu y+ a(x^2+y^2)y+bx^2y. \nonumber
\end{eqnarray}

There exist values of the parameters  for which  the system has additional symmetry and (\ref{ham-agk}) is completely integrable:

\begin{enumerate}
\item[i)] If $b=0$, by considering polar coordinates $x=r\cos\theta$ and $y=r\sin\theta$, the Hamiltonian (\ref{ham-agk}) becomes
\begin{equation}\label{ham-b0}
H(x,y)= \frac{p_r^2}{2}+\frac{p_\theta^2}{2r^2} -\frac{\mu}{2}r^2-\frac{a}{4} r^4,
\end{equation}
which does not depend on $\theta$. Since $p_\theta$ is a second independent first integral
the angular momentum $r^2\dot{\theta} = x\dot{y}-\dot{x}y$ is conserved.
\item[ii)] If $a=-b$, system (\ref{eq-ham}) reduces to two identical uncoupled Duffing's oscillators
$$\ddot{x} =  \mu x+ax^3 \quad \mbox{and} \quad
\ddot{y}  = \mu y+ay^3, $$
which can be integrated by quadrature, see \cite{wigg}.
\item[iii)] In case  $b=2a$, we employ the symplectic transformation
\begin{eqnarray*}
x=\frac{1}{\sqrt{2}}(u-v), &&  y=\frac{1}{\sqrt{2}}(u+v),\\
p_x=\frac{1}{\sqrt{2}}(p_u-p_v), && p_y=\frac{1}{\sqrt{2}}(p_u+p_v)
\end{eqnarray*}
and one  easily finds  that the Hamiltonian  (\ref{ham-agk}) becomes
\begin{equation}\label{ham-b2a}
H(u,v)=\frac{1}{2}(p_u^2+p_v^2)-\frac{\mu}{2}(u^2+v^2)-\frac{a}{2}(u^4+v^4).
\end{equation}
The transformed Hamiltonian  system decouples into two single degree of freedom oscillator $$
\ddot{u}-\mu u - 2a u^3=0 \quad \mbox{and} \quad   \ddot{v}-\mu v -2av^3=0,
$$
and each of the them  is integrable by quadrature.
\end{enumerate}

Let us remark that the first two cases were previously discovered by Armbruster et. al \cite{agk}. 
Elmandouh \cite{elmandouh}  also cited these two cases and added the case $b=2a$ when the  angular velocity $\omega=0$. 

\subsection{Application  of Bostan-Combot-Safey El Din algorithm}
The AGK Hamiltonian has the following polynomial potential which can be grouped into two homogeneous
potentials of grades two and four: 
\begin{equation}
V(x,y)=- \mu(x^2+y^2) +\frac{a}{4}(x^2+y^2)^2 -\frac{b}{2}x^2y^2=V_{min}+V_{max}.
\end{equation}
An old result of Hietarinta (\cite{hietarinta}, Corollary  of Theorem 3') proved again by Yoshida \cite{yoshida} and Nakagawa \cite{Naka} states the following.
\begin{proposition}\label{H}
Suppose that the Hamiltonian with a non-homogeneous polynomial potential,
$$H=\frac{1}{2}(p_1^2+p_2^2) + V_{min}(q_1, q_2) +  \dots + V_{max}(q_1, q_2)$$
admits an additional rational first integral $\Phi = \Phi (q_1,q_2,p_1,p_2)$.
Here $V_{min}$ is the lowest degree part of the potential and $V_{max}$ the highest degree part. 
Then, each of the two Hamiltonian,
$$H_{min} = (p_1^2+p_2^2)+V_{min}(q_1,q_2), \qquad  H_{max} = (p_1^2+p_2^2)+V_{max} (q_1,q_2) $$
also admits an additional rational first integral.
\end{proposition}

In our case $V_{min}=\mu(x^2+y^2)$ is obviously integrable for any value of $\mu$. It remains to determine for which values of the parameters  the quartic potential $V_{max}$ is
integrable. Homogeneous potentials have been worked before and after Morales-Ramis seminal theorem (\cite{combot}, \cite{maria}), in particular
Bostan et. al.  \cite{combot} take the potential to a polar form by making $x=r\cos{\theta}$, $y=r\sin{\theta}$ and reducing them to the form
\begin{equation}
V(r,\theta)=r^kF(e^{i\theta}).
\end{equation}

This change makes the evaluation of Morales-Ramis table  (Table \ref{table:MR}) for the necessary condition of integrability for homogeneous potentials easier.  Bostan et al. are able
to built a nice Maple algorithm using the following theorem.
\begin{theorem}(Bostan et. al.\cite{combot})
Let $V \in \mathbb{C}(x,y)$ be a homogeneous potential with polar representation $(F, k)$ and let $\Lambda$ be the following set
$$\tilde{\Lambda}(F,k):=\left \{ k-\frac{z^2F''(z) }{F(z)} : z\ne 0, F'(z)=0, F(z)\ne 0 \right\}. $$

Let $Sp_D (\Delta^2V )$ denote the union of the sets $Sp(\Delta^2V (c))$ taken over all Darboux points $c \in D$ of $V$,  then
$${k(k-1)}\cup Sp_D(\Delta^2V) = {k(k-1)}\cup{\tilde\Lambda}. $$  Moreover, if $V$ is integrable, then 
$\tilde{\Lambda} \in E_k$,
where $E_k$ is the so called Morales-Ramis table. 
\end{theorem}

This result allows the construction of a powerful algorithm which
 allows quick evaluation of integrability even to homogeneous potential of degree 9.

Applying the Bostan-Combot-Safey El Din algorithm given in \cite{combot}, we have found that the necessary values for integrability are $b=0$, $b=2a$ and $b=-a$ which
are the same of the well known integrable cases described at the beginning of this section.  The algorithm also inform us the spectrum of the
Hessian matrix for each necessary condition of rational integrability, yields
\begin{itemize}
\item $b=0$, the Darboux points have Hessian spectrums $[12,4]$,
\item $b-a=0$, the Darboux points have Hessian spectrums  $[12,12]$ or $[12,0]$,
\item $a+b=0$, the Darboux points have Hessian spectrums $[12,12]$ or $[12,0]$. \end{itemize}
We note that the trivial eigenvalue of the Hessian $\lambda=12$ appears also as a non trivial eigenvalue.

Now applying Theorem \ref{H} since the harmonic oscillator is trivially integrable we get the necessary condition for the integrability of the quartic
potential. The procedure described at the beginning of this section showing the separability of the potential for the classical values implies the following result.
\begin{theorem}\label{ratint}
The Armburster-Guckenheimer-Kim potential (\ref{ham-agk}) is rationally integrable if and only if $b = 0$, $b=2a$ and $b=-a$. 
\end{theorem}

Note that for $\mu =0$ we have an homogeneous quartic potential  whose 
necessary conditions for rational  integrability are  the same values given in Theorem \ref{ratint}.

\subsection{Application of Morales-Ramis Theorem}
The Hamiltonian has the square symmetry, that is, it is invariant under a rotation by $\pi/2$, and is also time reversible.
So, the planes of symmetry are given by  planes which are invariant  under the action of the dihedral group $D_4$:
\begin{itemize}
    \item $\Gamma_1=\{(x,y,p_x,p_y):y=p_y=0\}$
    \item $\Gamma_2=\{(x,y,p_x,p_y):x=p_x=0\}$
    \item $\Gamma_3=\{(x,y,p_x,p_y):y=x,p_y=p_x\}$
    \item $\Gamma_4=\{(x,y,p_x,p_y):y=-x,p_y=p_x\}$
    \item $\Gamma_5=\{(x,y,p_x,p_y):y=x,p_y=-p_x\}$
    \item $\Gamma_6=\{(x,y,p_x,p_y):y=-x,p_y=-p_x\}$.
\end{itemize}

Let us observe that if we restrict  Hamiltonian (\ref{ham-agk}) to the invariant planes of symmetry
$\Gamma_{1,2}$ and  respectively $\Gamma_{3,4,5,6}$, we get the same form.  
Thus,  it is enough to  consider the following two cases:
\begin{subequations}
\begin{equation}\label{ham-ga1}
\displaystyle{h={p_x^2\over 2}-{\mu\over 2}x^2-{a\over 4}x^4, \qquad H\vert_{\Gamma_1}=h},
\end{equation}
\begin{equation}\label{ham-ga3}
\displaystyle{\widehat{h}={p_x^2\over 2}-{\mu\over 2}x^2-{2a+b\over 4}x^4, \qquad H|_{\Gamma_3}=2\widehat{h}}.
\end{equation}
\end{subequations}

We note that each one of these expressions  correspond to a Hamiltonian with one degree of freedom,
and also setting change $a\mapsto 2a+b$  into
(\ref{ham-ga1})  we obtain (\ref{ham-ga3}). So, along the paper we will  compute the conditions of
integrability of the {\em normal variational equations} over  $\Gamma_1$ and $\Gamma_3$, and also
we substitute  $a\mapsto 2a+b$ into the particular solution $x(t)$ obtained in $\Gamma_1$ to get them in $\Gamma_3$.

 In order to apply the Morales-Ramis theory  our strategy  consists in selecting  a non-equilibrium particular solution in the invariant planes  $\Gamma_1$  and $\Gamma_3$.
The following step is to obtain an integral curve for the Hamiltonian (\ref{ham-ga1}).
Thus, to perform our study we take into account that
a Hamiltonian is a conservative system and we can fix energy level for $\Gamma_1$, $h$; that is,
\begin{equation}\label{ham-work}
{p_x^2\over 2}-{\mu\over 2}x^2-{a\over 4}x^4=h.
\end{equation}

To solve it we rewrite (\ref{ham-work}) as
\begin{equation}\label{pareqh1b}
\displaystyle{\frac{dx}{dt}=\pm\sqrt{2h+\mu x^2+{a\over 2}x^4}}.
\end{equation}

We observe that Eq. \eqref{pareqh1b} for any value of $h$, $a$ and $\mu$ corresponds to the well known \emph{incomplete elliptic integral of first kind} (see \cite{abramowitz}), whose solutions are no always 
{\em elementary functions}.
In order to obtain a manageable equation with a not too
complicated base  fields, and get explicit solutions
for Eq. \cref{pareqh1b},  we should 
restrict the values of $h$.
Furthermore, to be able to compute Galois groups for the variational equations along such particular 
solutions,  we should stay with these solutions. 
In particular we choose the energy level  $h=0$, then the restricted last equation becomes
\begin{equation}\label{pareqh1}
\displaystyle{\frac{dx}{dt}=\pm\sqrt{\mu x^2+{a\over 2}x^4}}.
\end{equation}

Now, observe that  (\ref{pareqh1}) has  three equilibrium points given by
$\displaystyle{x=0,\pm i\sqrt{\frac{2\mu}{a}}}$.
To avoid triviality, we exclude these points and assume that $x(t)$ is not a constant.

In a similar way, to obtain algebraic solutions over a suitable differential field  in equation \eqref{pareqh1b} with $ah\neq 0$, we can choose the energy level $\displaystyle{h={\mu^2\over 4a}}$. Thus, the restricted equation becomes 
$$\pm\frac{dx}{dt}={\mu\over \sqrt{2a}}+ \sqrt{a\over 2}x^2.$$ Observe that the last equation has two equilibrium points given by $\displaystyle{x=\pm i\sqrt{\frac{\mu}{a}}}$. To avoid triviality, we exclude these points and assume that $x(t)$ is not a constant. Recall that similar results are obtained for the respective $\hat{h}$ in $\Gamma_3$.

Next, we  present a rigorous proof  of the  non-integrability of Hamiltonian (\ref{ham-agk})
 using the Morales-Ramis theory. Since this theorem gives conditions for integrability in  terms of the differential Galois group 
of  the variational equations along a particular solution,
we  compute the  linearized system of (\ref{eq-ham})  along a non-constant solution of the dynamical system  given above.

In a general frame we assume $\gamma (t)= (x(t), y(t), p_x(t), p_y(t))\in\Gamma_{1,3}$  is a non-stationary solution of (\ref{eq-ham}), then the variational equations $\dot{\xi}=A(t)\xi$ of
(\ref{eq-ham}) along $\gamma (t)$ are given by
\begin{equation*}
\left( \begin{array}{c}\dot{\xi}_1 \\ \dot{\xi_2} \\ \dot{\xi_3} \\  \dot{\xi_4} \end{array}\right)
=\begin{pmatrix}
0 & 0 & 1 & 0\\ 0 & 0 & 0 & 1\\ \mu + 3ax(t)^2 + (a+b) y(t)^2   & 2(a+b)x(t)y(t) & 0 & 0\\ 
2(a+b)x(t)y(t)  & \mu +(a+b)x(t)^2 + 3a y(t)^2 & 0 & 0
\end{pmatrix}
\left( \begin{array}{c}  \xi_1 \\  \xi_2  \\  \xi_3  \\  \xi_4 \end{array} \right).
\end{equation*}
%
Assume $\gamma (t)= (x(t), 0, p_x(t), 0)\in\Gamma_1$  is a non-stationary solution of (\ref{eq-ham}), then the variational equations $\dot{\xi}=A(t)\xi$ of
(\ref{eq-ham}) along $\gamma (t)$ are given by
\begin{equation*}
\left( \begin{array}{c}\dot{\xi}_1 \\ \dot{\xi_2} \\ \dot{\xi_3} \\  \dot{\xi_4} \end{array}\right)
=\begin{pmatrix}
0 & 0 & 1 & 0\\ 0 & 0 & 0 & 1\\ \mu +3ax(t)^2 & 0 & 0 & 0\\ 0 & \mu +(a+b)x(t)^2 & 0 & 0
\end{pmatrix}
\left( \begin{array}{c}  \xi_1 \\  \xi_2  \\  \xi_3  \\  \xi_4 \end{array} \right).
\end{equation*}
The second order variational system is composed 
of two uncoupled Schr\"odinger equations 
\begin{equation}\label{eqvar1}
\ddot\xi_1=\left(3ax^2(t)+\mu\right)\xi_1,
\end{equation}     
\begin{equation}\label{eqvar2}
\ddot\xi_2=\left((a+b)x^2(t)+\mu\right)\xi_2.\end{equation}  
It is well known that the tangential variational equations do not provide conditions of non-integrability of Hamiltonian systems with two degrees of freedom, see \cite{mora1}. It is easy to see that when we replace  $2a$ instead $b$ into 
Eq. (\ref{eqvar2}),  the variational normal and tangential equations are the same.  For this reason, from now on  we let out the tangential equations.

We now assume $\gamma (t)= (x(t), x(t), p_x(t), p_x(t))\in\Gamma_3$  is a non-stationary solution of (\ref{eq-ham}), then the variational equations of
(\ref{eq-ham}) along $\gamma (t)$ becomes
\begin{equation*}
\left( \begin{array}{c}\dot{\xi}_1 \\ \dot{\xi_2} \\ \dot{\xi_3} \\  \dot{\xi_4} \end{array}\right)
=\begin{pmatrix}
0 & 0 & 1 & 0\\ 0 & 0 & 0 & 1\\ \mu + (4a+b)x(t)^2 & 2(a+b)x(t)^2 & 0 & 0\\ 2(a+b)x(t)^2 & \mu +(4a+b)x(t)^2 & 0 & 0
\end{pmatrix}
\left( \begin{array}{c}  \xi_1 \\  \xi_2  \\  \xi_3  \\  \xi_4 \end{array} \right).
\end{equation*}
In this case the normal equation yields
\begin{equation}\label{eqvar2b}
\ddot\xi_2=\left((2a-b)x(t)^2+\mu\right)\xi_2.\end{equation}  

The computation is  carried out considering several combinations of the signs of the parameters $\mu$ and $a$ with
$\mu\neq 0$,  and distinguishing between cases $a=0$ and $a\neq 0$.  
We also assume  without loss of generality that the constants of integration of the normal variational equations are zero.

 \subsection{Case $a=0$: rational non-integrability}
 In this case the normal variational equation  (\ref{eqvar2})  takes the form
\begin{equation}\label{eqvar2-a0}
\ddot\xi_2=\left(bx^2(t)+\mu\right)\xi_2,\end{equation} 
respectively.

A simple computation shows that  $x(t)=\cos(\sqrt{\mu}t)$ with $h= -1/2$ is a particular solution of the variational equations on the symmetry plane $\Gamma_1$.
Therefore we  obtain that (\ref{eqvar2-a0}) along  $x(t)$ takes the form
$$\ddot\xi_2=\left(b\cos^2(\sqrt{\mu}t)+\mu\right)\xi_2.$$
Taking into account that $\cos^2(\sqrt{\mu}t)={1\over2}(1+\cos(2\sqrt{\mu}t)$,  the last equation  becomes
\begin{equation}\label{mathie1}
\ddot\xi_2=\left(\frac{b}{2}+\mu +\frac{b}{2}\cos(2\sqrt{\mu}t) \right)\xi_2\end{equation}     
which is a Mathieu equation.   
In a similar way, we obtain an equivalent result for $\Gamma_3$ by replacing $b \mapsto -b$.

 We obtain for $b\neq 0$ and $\mu\neq 0$, according to Theorem \ref{acbl}, that Eq. (\ref{mathie1}) 
 is non-integrable, see \cite{acbl}.
 Therefore,  we have proved the following result.
 \begin{theorem}
The Hamiltonian system \eqref{eq-ham}, with $a=0$, $b\neq 0$  and $\mu\neq 0$ 
is non-integrable through rational first integrals.
 \end{theorem}  

 Now in the case $\mu=0$, $a=0$, $h=0$ and $\hat{h}\neq 0$
 the normal variational Eq. (\ref{eqvar2}) and Eq. (\ref{eqvar2b}) are no-integrables 
 because they correspond to quantic harmonic oscillators with zero energy level for both planes
 $\Gamma_1$ and $\Gamma_3$, see
 \cite{acbl,acmowe}. Finally, the cases $h=\hat{h}=0$ are not considered because we obtain for both planes
 $\Gamma_1$ and $\Gamma_3$ that the normal variational equations (\ref{eqvar2}) and  (\ref{eqvar2b}) are trivially integrables. The first one corresponds to a differential equation with constant coefficients and the second one corresponds to a Cauchy-Euler equation, which are always integrable with abelian differential Galois group, 
 see \cite{acbook,acmowe}.
 Thus Morales-Ramis Theorem does not  give any obstruction to  the integrability.

\subsection{Case $a\neq 0$: meromorphic non-integrability}
 We consider $\gamma (t)= (x(t), 0, p_x(t), 0)\in\Gamma_1$  which is a non-stationary solution of
(\ref{eq-ham}). Next, we solve Eq. \eqref{pareqh1} for $h=0$ and $\displaystyle{h={\mu^2\over 4a}}$ which give us 
possible particular solutions over $\gamma (t)$, shown in the Table \ref{tabla2}. 
\medskip

\begin{table}[ht]
\begin{center} 
\begin{tabular}{|c|c|c|} 
\hline
     & $h=0$ & $\displaystyle{h={\mu^2\over 4a}}$ \\ \hline
    1. & $\displaystyle{x(t)=\pm i\frac{\sqrt{\frac{2\mu}{a}}}{\cosh(\sqrt{\mu}t) } }$ 
    & $x(t)=\pm\displaystyle{i\sqrt{\frac{\mu}{a}}\tanh\left(\sqrt{\frac{\mu}{2}}t\right)}$ \\ \hline
    2. & $\displaystyle{x(t)=\pm{i\sqrt{2\mu\over a}\over\cos(i\sqrt{\mu}t)}}$ 
& $x(t)=\pm\displaystyle{i\sqrt{\mu\over a}\tan\left(i\sqrt{\mu\over 2}t\right)}$ \\ \hline
    3. & $\displaystyle{x(t)=\pm{\sqrt{2\mu\over a}\over
\sin\left(i\sqrt{\mu}t\right)} }$ & $x(t)=\pm\displaystyle{\sqrt{\mu\over a}\textrm{cot}\left(i\sqrt{\mu\over 2}t\right)}$ \\ \hline
  4. & $\displaystyle{ x(t)=\pm{\sqrt{2\mu\over a}\over
\sinh\left(\sqrt{\mu}t\right)}}$
 & $x(t)=\displaystyle{\pm\sqrt{\mu\over a}\textrm{coth}\left(\sqrt{\mu\over 2}t\right)}$ \\ \hline
\end{tabular} 
\end{center}
\caption{Particular solutions over $\gamma (t)$.}\label{tabla2}
\end{table}

These solutions are complex functions where also $t$ is a complex number. 
Moreover, we note that is not possible to get algebraic particular solutions for $h\neq 0$ with $\mu=0$, thus we exclude this case. So, we consider only one algebraic particular solution, given in the previous table, for each one of the energy levels ($h=0$, $h\neq 0$), to construct the normal variational equations. The remaining cases are obtained
through the transformations $t\mapsto it$ and $t\mapsto \frac{\pi}{2}-t$. Furthermore, the particular solutions over $\Gamma_3$ can be obtained by the change $a\mapsto 2a+b$ into the particular solutions over $\Gamma_1$.

We recall that ${\mathcal T}_\ell$, the $\ell$-th triangular number, is  
defined by
$${\mathcal T}_\ell = 1 + 2 + \dots + \ell = \frac{1}{2}\ell (\ell+1),  \qquad \ell\in\mathbb{Z}^+.$$

Since $a\in\mathbb{C}^*$, instead of working with the space of parameters
$\{(a,b,\mu)\in\mathbb{C}^3\}$ we work with  $\mathcal W:=\{\kappa\in
\mathbb{C}: \, \kappa=b/a, a\neq 0\}$.
We define the following subsets of $\mathcal W$ with $\mu\neq 0$ given as:  $$\Lambda_1:=\{
\kappa: \kappa= \mathcal T_\ell -1\},\qquad  
\Lambda_2:=\left\{\kappa: \kappa= 2\frac{1-\mathcal{T}_{\ell}}{1+\mathcal{T}_{\ell}}\right\}.$$

We define by $\Lambda:=\Lambda_1\cap\Lambda_2$ to introduce our main
result.
\begin{theorem} Assume $a\mu\neq 0$. If $\kappa\notin \Lambda$ then the Hamiltonian (\ref{ham-agk}) is 
non-integrable through meromorphic first integrals.
\end{theorem}

\begin{proof}
%
We consider only the case $h=\hat{h}=0$ in the invariant planes $\Gamma_1$ and $\Gamma_3$, respectively.

We begin with  the variational equation near  $\Gamma_1$. 
We use (\ref{ham-ga1}) with $h=0$ to get $x(t)$ which is given in the Table \ref{tabla2}.  Substituting the expression for 
$x(t)$ in (\ref{eqvar2}) with $\kappa=b/a$, we obtain
\begin{eqnarray}
\ddot\xi_{21}=\mu\left(-{2(1+\kappa)\over
\cosh^2\left(\sqrt{\mu}t\right)}+1\right)\xi_{21}\label{ev21},
\end{eqnarray}
where we assume coefficients in the differential field
$\mathbb{K}=\mathbb{C}(\tanh(\sqrt{\mu} t))$.
Through  the change $\tau=\sqrt{\mu}t$ and letting
$\displaystyle{\xi'=\frac{d\xi}{dz}}$, \eqref{ev21} is transformed into the Schr\"odinger equation with P\"oschl-Teller potential
\begin{equation}\label{integra-legendre}
\xi_{21}''=\left(-{2(1+\kappa)\over
\cosh^2\left(\tau\right)}+1\right)\xi_{21},
\end{equation}
which is transformed into the Legendre's differential equation with parameters $\ell$ and $\tilde{\mu}$ 
such that $2(1+\kappa)=\ell(\ell + 1)$ and $\tilde{\mu}=\pm1$. Since the second parameter is an integer, by Theorem \ref{almp}, we know that Eq. \eqref{ev21}  is integrable if only if
 $1\pm \ell \in \mathbb{Z}$ or $\ell\in \mathbb{Z}$, because we are only getting 
 the case 1 of Theorem \ref{almp}.
 So, $\ell (\ell
+1)$ must be  an integer, which implies that Eq. \eqref{ev21} has Liouvillian solutions if and only if $\kappa\in  \Lambda_1$. Moreover, the differential Galois group $\mathrm{DGal}(\mathbb{L}/\mathbb{K})$ of this normal variational equation Eq. \eqref{ev21} for $\kappa\in\Lambda_1$ is isomorphic to the additive group $\mathbb{C}$, which is a connected abelian group; while for $\kappa\notin \Lambda_1$ the differential Galois group $\mathrm{DGal}(\mathbb{L}/\mathbb{K})$ is an algebraic subgroup isomorphic to $\mathrm{SL}(2,\mathbb{C})$ and thus non abelian. Therefore, by Morales-Ramis Theorem we can conclude that if $\kappa\notin \Lambda_1$ then the Hamiltonian AGK is not integrable through meromorphic first integrals.

The normal variational equations for the second, third and fourth cases of the particular solutions $x(t)$ are transformed into  Eq. (\ref{ev21}) through the combination of the changes of variables $t\mapsto i t$ and  $t\mapsto \frac{\pi}{2}-t$.  
So we get similar results of integrability of Eq. \eqref{ev21} and  consequently  similar conditions of non-integrability 
for the Hamiltonian AGK. 

From here we consider the variational equation near  $\Gamma_3$.  
We use  (\ref{ham-ga3})  to obtain $x(t)$  which is for $\hat{h}=0$ given in the 
Table \ref{tabla2} with $2a+b$ instead of $a$. Now we substitute $x(t)$ 
 into the normal equation (\ref{eqvar2b}). It has the form
\begin{eqnarray}
\ddot\xi_{21}=\mu\left(-{2(2-\kappa)/(2+\kappa)\over
\cosh^2\left(\sqrt{\mu}t\right)}+1\right)\xi_{21}\label{ev21_3}.
\end{eqnarray}
As previously, we rescaled the independent variable  to get a Schr\"odinger equation with P\"oschl-Teller 
potential, which can be transformed into a Legendre's differential equation.
From (\ref{pteq}), we can conclude that   
$\displaystyle{\frac{2-\kappa}{2+\kappa} =  {\mathcal T}_\ell}$ or equivalently
$\displaystyle{\kappa = 2\frac{1-{\mathcal T}_\ell}{1+{\mathcal T}_\ell}}$.
A similar analysis applied  to the above case  leads to the  non-integrability conditions 
 through meromorphic first integrals for  Hamiltonian AGK with  $\kappa \notin \Lambda_2$.

Finally,  we conclude that the Hamiltonian AGK (\ref{ham-agk}) is non-integrable through  meromorphic first integrals for 
 $\kappa\notin \Lambda$.
\end{proof}

We observe that $\Lambda = \Lambda_1 \cap \Lambda_2 = \{-1,0,2\}$. Furthermore, the values $b=0, 2a, -a$ given by  the Theorem \ref{ratint}  correspond to 
 $\kappa = 0,2,-1$, respectively. Then, the former  satisfy   the necessary condition for  meromorphic integrability.

We would now like to make some remarks.
In order to apply Morales-Ramis theory as non-integrability criterium, it is 
enough to exhibit an orbit for which the normal variational equation is not abelian. 
Taking this into account and since 
$\widetilde{\mu}$ is not constant,  we have only been considering  the proof 
of the last theorem for the both cases $h=\hat{h}=0$. 
Moreover, if we replace any particular solution for $\displaystyle{h={\mu^2\over 4a}}$  into the normal variational equation, Eqs. \eqref{eqvar2} 
and carries out combinations of the change of variables $t\mapsto i t$ and  $t\mapsto \frac{\pi}{2}-t$,  we obtain the following normal variational equation
\begin{equation}\label{integra-legendre2}
\xi_{22}''=\left(-{2\left(1+\kappa\right)\over
\cosh^2\left(\tau\right)}+2\kappa+1\right)\xi_{22},
\end{equation}
which is transformed into the Legendre's differential equation with parameters $\ell$ and $\tilde{\mu}$ such that $2(1+\kappa)=\ell(\ell + 1)$ and $\tilde{\mu}^2=2\kappa+1$. We get equivalent conclusions for $\displaystyle{\hat{h}={\mu^2\over 4(2a+b)}}$ in the variational equation (\ref{eqvar2b}).

Finally, we remark that  for $\mu=0$ and $h=\hat{h}=0$, the normal variational equations (\ref{eqvar2}) and  (\ref{eqvar2b}) correspond to the Cauchy-Euler equation for $\Gamma_1$ and $\Gamma_3$, respectively,  given by
\begin{eqnarray*}
\ddot\eta_1=\frac{2(1+\kappa)}{t^2}\; \eta_1, \qquad\qquad
\ddot\eta_3=2\frac{2-\kappa}{(2+\kappa)t^2}\; \eta_3,
\end{eqnarray*}
which are always integrable with abelian differential Galois group, 
see \cite{acbook,acmowe}.
Thus Morales-Ramis Theorem does not  give any obstruction to  the integrability.
Furthermore for $h\neq 0$ and $\hat{h}\neq 0$ the particular solution is an elliptic integral of
first kind that is  not algebraic. In this case
the variational equations are not in a suitable differential field to
apply Morales-Ramis Theorem. 

\section{Numerical experiments}\label{sec-poincare}
In this section, we present a quick view of the system behaviour through 
the  Poincar\'e section technique.  We have chosen the Poincar\'e section  ($y=0$, $p_y > 0$). The values of the parameters 
were chosen as to view the system behaviour at the well known integrable cases, and then by moving   the values of one
of the parameters ($a$ or $b$),
carry the system to a region were KAM tori start being destroyed. The energy and $\mu$ are maintained constant for each set
of experiments.

The first example stemmed from the search of Armbruster et. al  reported in Figure 5 of  \cite{agk} which has the following set of parameters: $h=5.7$,  $\mu= - 5$ , $a = 1$ and $ b = 0.5$. We here have worked with $b$ going from zero to five. The Figures \ref{fig:1} and \ref{fig:2} show the evolution from the integrable case $a=1$ and $b=0$ to $a=1$ and $b=0.5$ where the superior part of the section is completely chaotic. This is the case of Figure 5  of  \cite{agk} where they have shown the chaotic layer only. The intermediary figures show how a meagre top and bottom region of KAM
tori, $b=0.1$,  evolves to partially chaotic region ($0. 3$, and finally to regions of  no tori at all since trajectories  become unbounded. Some KAM tori remain in the horizontal regions  which we show in Figure \ref{fig:2}.

In Figure \ref{fig:3}  the integrable case $a=1$, $b=-1$ with $h=3.5$ and $\mu=5$  is shown. Note now that the section of the tori
have rotated $\pi/2$ and this indicates a different way of breaking down. In fact, in Figure \ref{fig:3},  $b=2.8$ and keeping the remaining parameters the same, we see a stable periodic orbit at the origin and two unstable ones at each side of the figure.  Finally for $b=6$: the  top and
bottom cuts of the KAM tori break down with the internal trajectories escaping to non--closed regions of the phase space.  This behaviour can be seen in Figure \ref{fig:3}. So far, although the route to chaos seems different for the two cases, it seems that scape trajectories play an important role.

The last integrable case is $b=-a$ can be seen in Figures \ref{fig:4}. We have chosen as the starting value $a=1$ and $b=1$, the remaining values are $h=3.5$ and $\mu=5$. The cuts of the Poincar\'e section start centered and for $a=1.8$ there is  a horizontal
bifurcation and the orbits in the central region starts to scape, Figure \ref{fig:4} (middle). For $a=2.05$ the central region is more eroded and all the bordering trajectories escape, Figure \ref{fig:4} (bottom).

Finally an example o chaotic behaviour that contradicts all the previous experiments.  As can be seen in Figures \ref{fig:5}, there is no escaping orbits among the chaotic ones. We started with $a=1$, $b=-1$ and the value of the Hamiltonian $h=0.2$ and 
$\mu=1$ maintained constant through the evolution of $b$ from $-1$ to $-5$. The section in $a=1$ and $b=-1$ is similar to
Figure \ref{fig:4} (top). When $b=-2.5$ there is  already  a major bifurcation in the vertical direction.  Finally when $b=-5$ there is chaos
at large with some islands still left. Apparently the value of $\mu$ is responsible for this behaviour.

\begin{figure}[tbhp]
\centering
\includegraphics[scale=0.85]{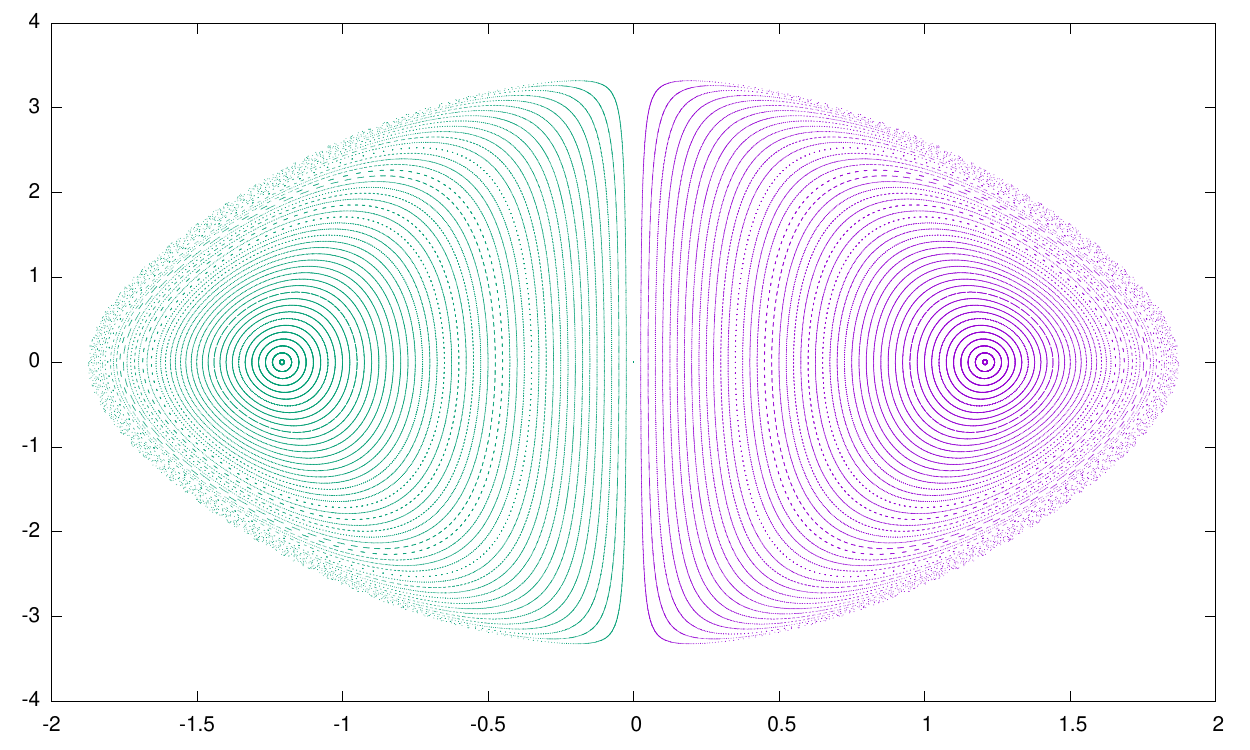}
\includegraphics[scale=0.85]{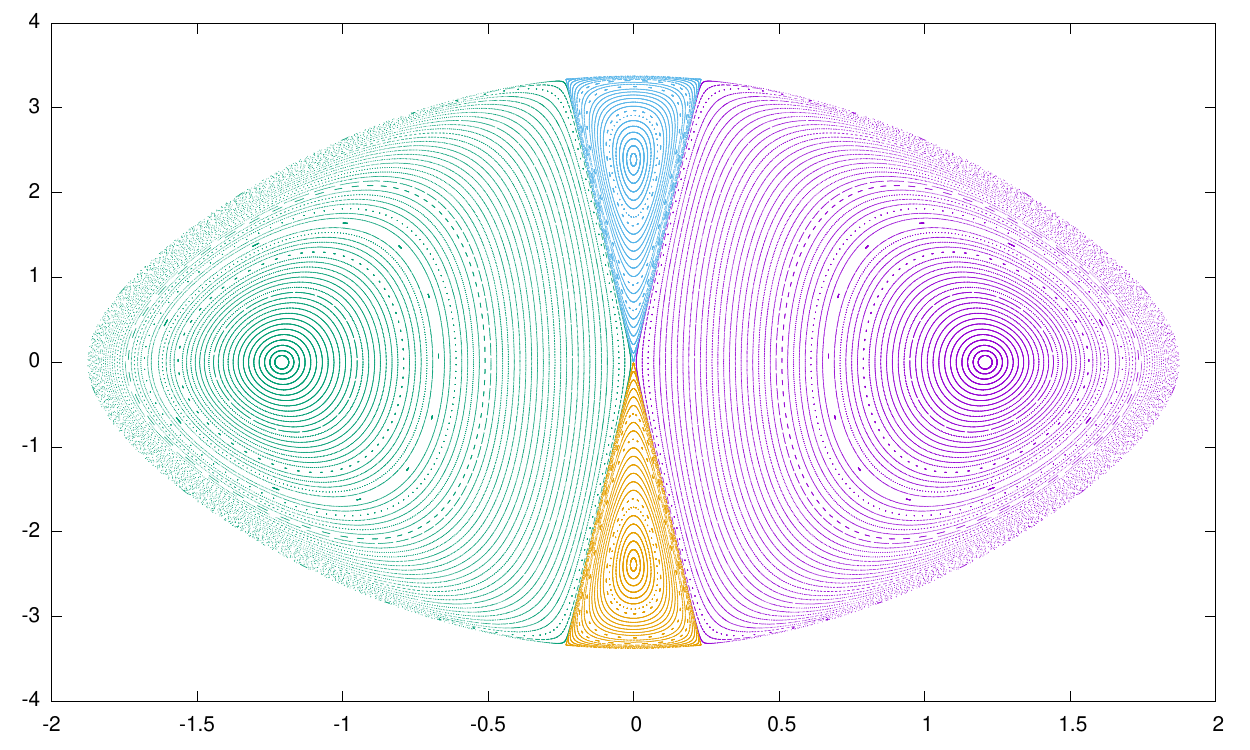}
\caption{Poincar\'e section $(x,p_x)$:  The top figure shows the integrable case  $b=0$, $a=1$, $\mu=5$ and $h=5.7$, while 
in the bottom figure is showing  the same values of the parameters as last one with the exception of  $b= 0.01$. Note the wedge of KAM tori which appears in the vertical direction.}
\label{fig:1}
\end{figure}

\begin{figure}[ht]
\begin{center} 
\includegraphics[scale=0.85]{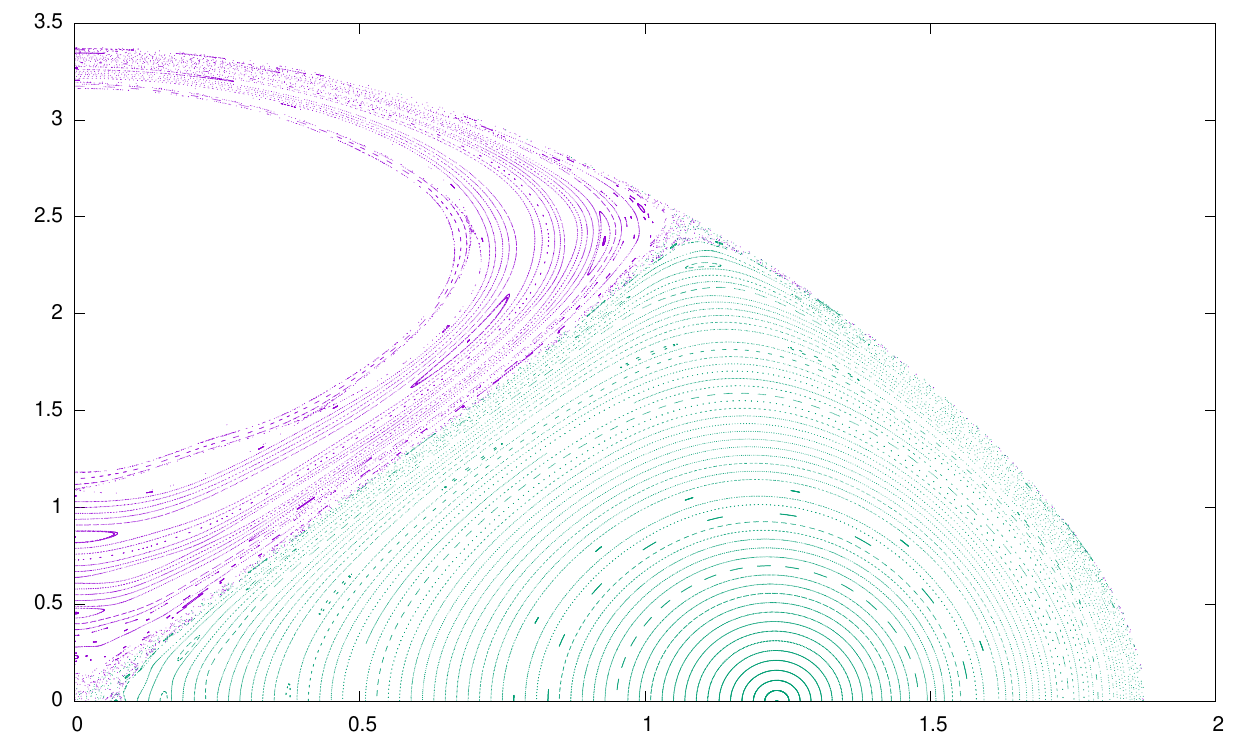}
\includegraphics[scale=0.85]{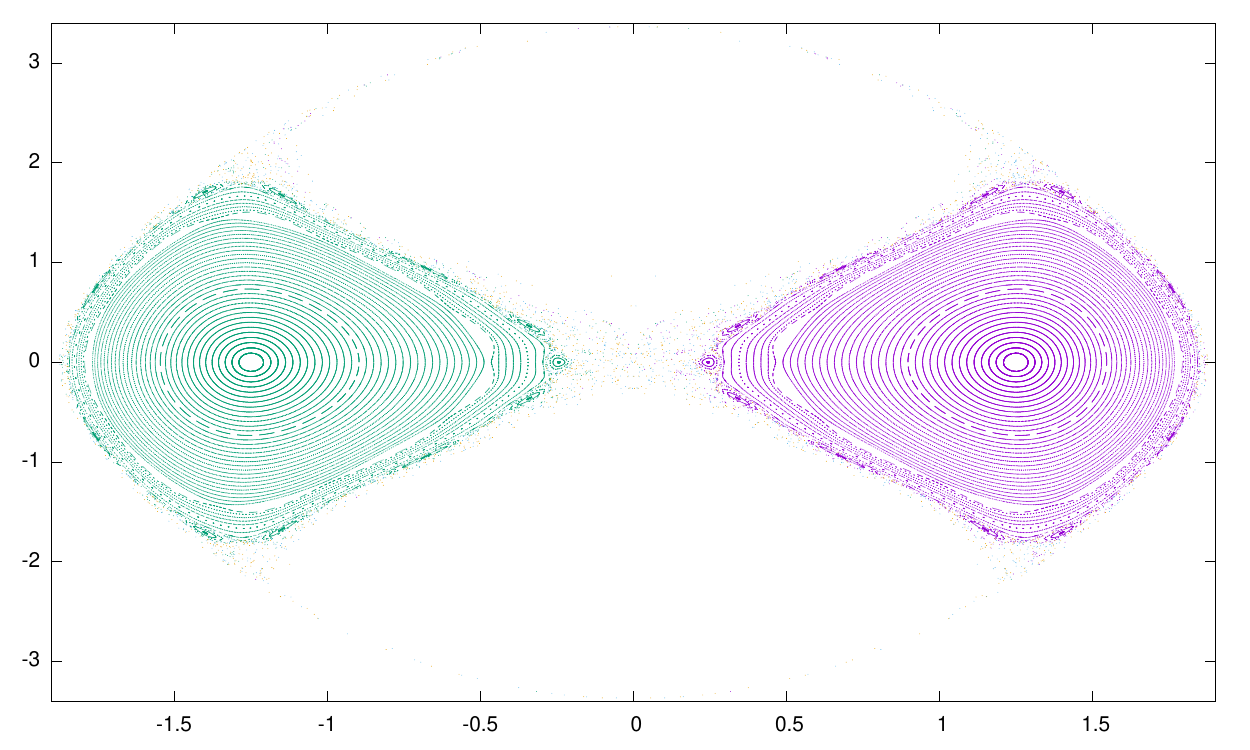}
\end{center}
\caption{Poincar\'e section $(x,p_x)$:  $b=0.3$ and remaining parameters as in Figure \ref{fig:1} (top); The vertical tori break through scape in the open phase space; we reach  $b=0.5$ (bottom), and now we have the same conditions as Armbruster et. al \cite{agk} (Figure 5); the horizontal empty region show by Armbruster et. al \cite{agk} is filled with remnant tori.}
\label{fig:2}      
\end{figure}

\begin{figure}
\begin{center}
	Pictures available on the original file
\end{center}
\caption{Poincar\'e section $(x,p_x)$:  the integrable  case $b=2a$ with   $a=1$, $b=2$, $\mu=5$,  $h=2$ (top); Note that the tori have rotated to the vertical axis;  for $b=2.8$ there is a major bifurcation in the vertical direction (middle) and  for  $b=6$ there appears chaos in the escaping trajectories (bottom).}
\label{fig:3}      
\end{figure}

\begin{figure}
\begin{center}
\includegraphics[scale=0.9]{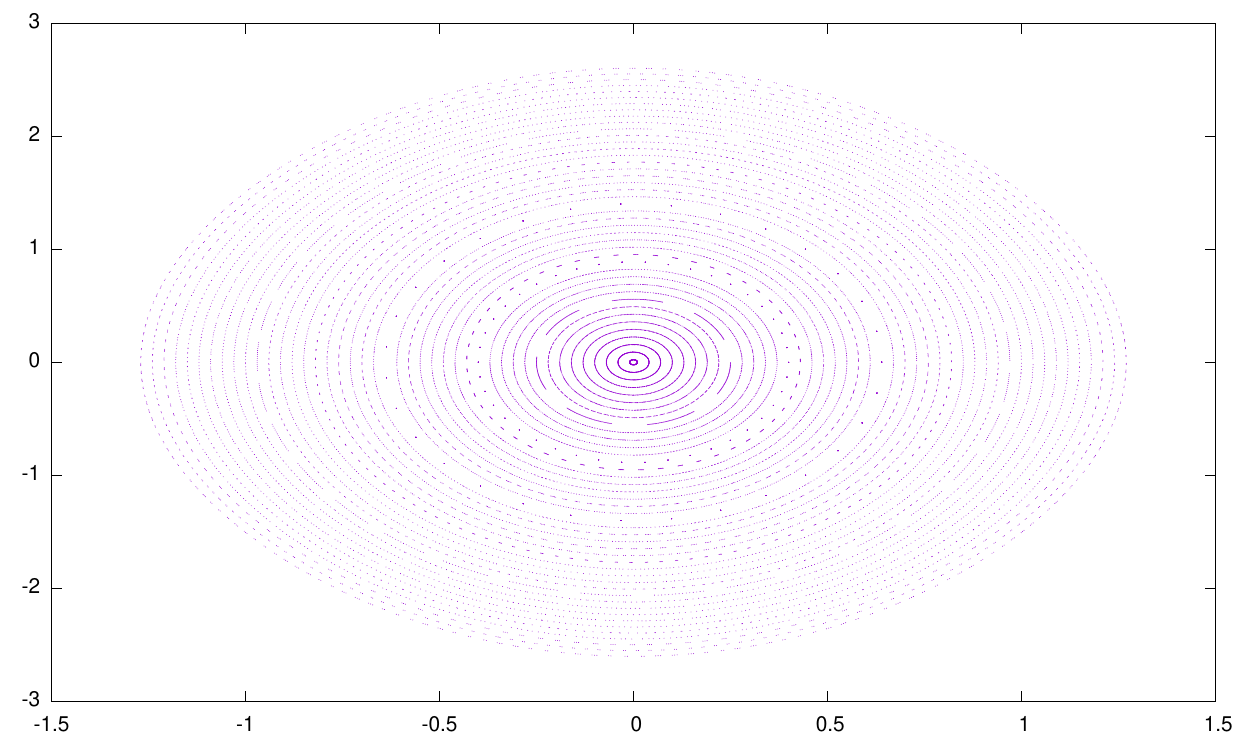}
\includegraphics[scale=0.9]{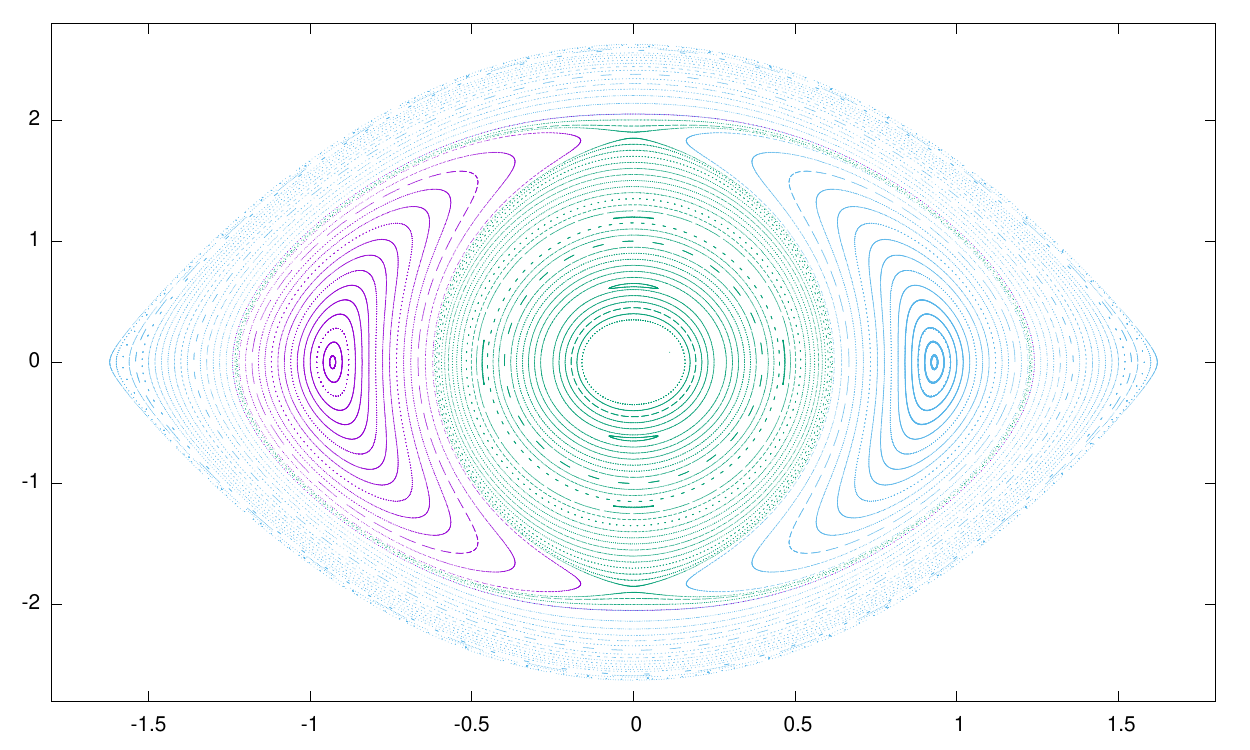}
\includegraphics[scale=0.9]{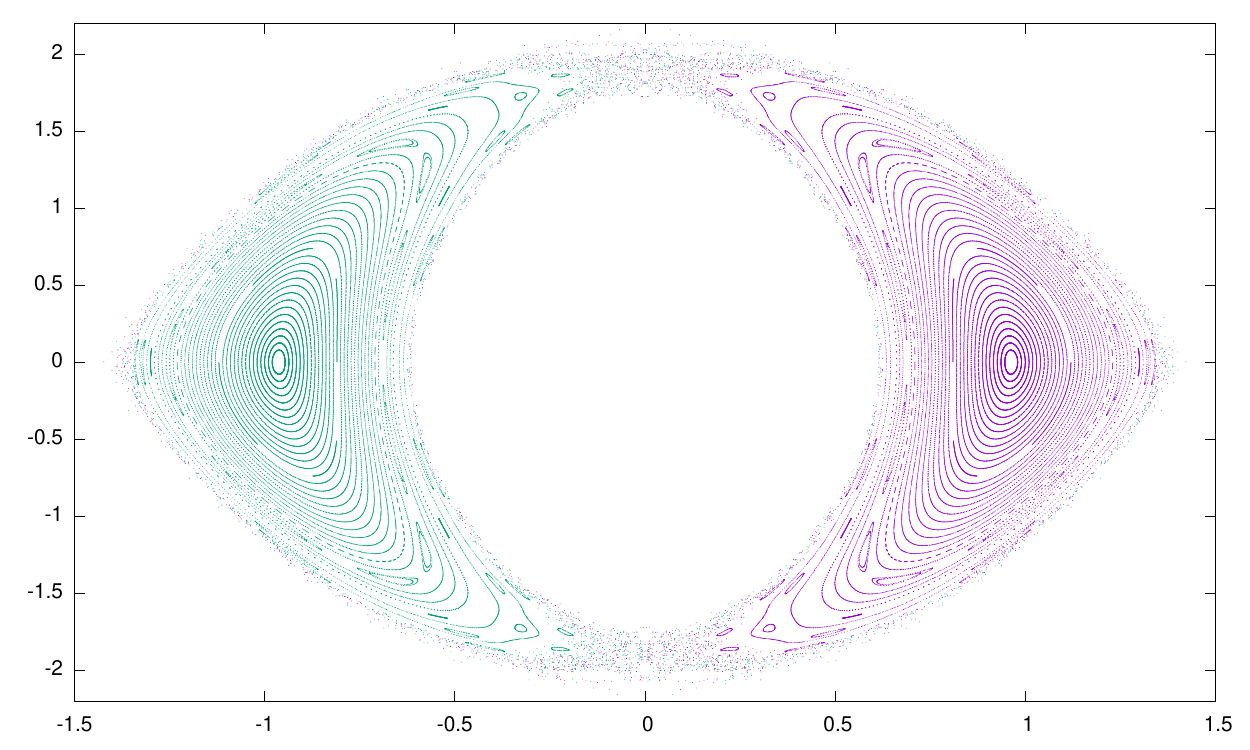}
\end{center}
\caption{ Poincar\'e section $(x,p_x)$: the integrable case $b=-a$ with  $a=1$, $b=-1$, $h=3.5$ and  $\mu=5$ (top) and we can see the symmetrical KAM tori; taking $a$ to $a=1.8$ (middle) there is a horizontal bifurcation  and the central tori start escaping; with $a=2.08$ (bottom) the central part becomes more chaotic and most of the bordering  trajectories escape. }
\label{fig:4}      
\end{figure}

\begin{figure}
\begin{center}
\includegraphics[scale=0.9]{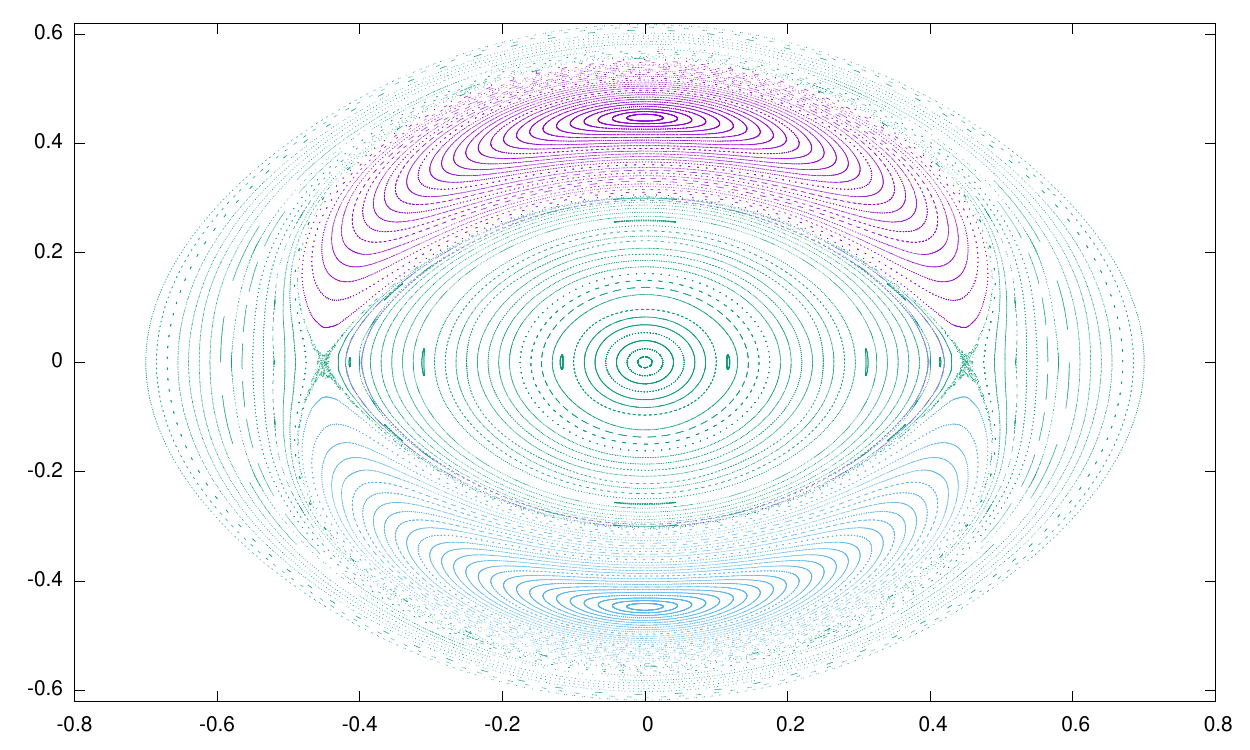}
\includegraphics[scale=0.9]{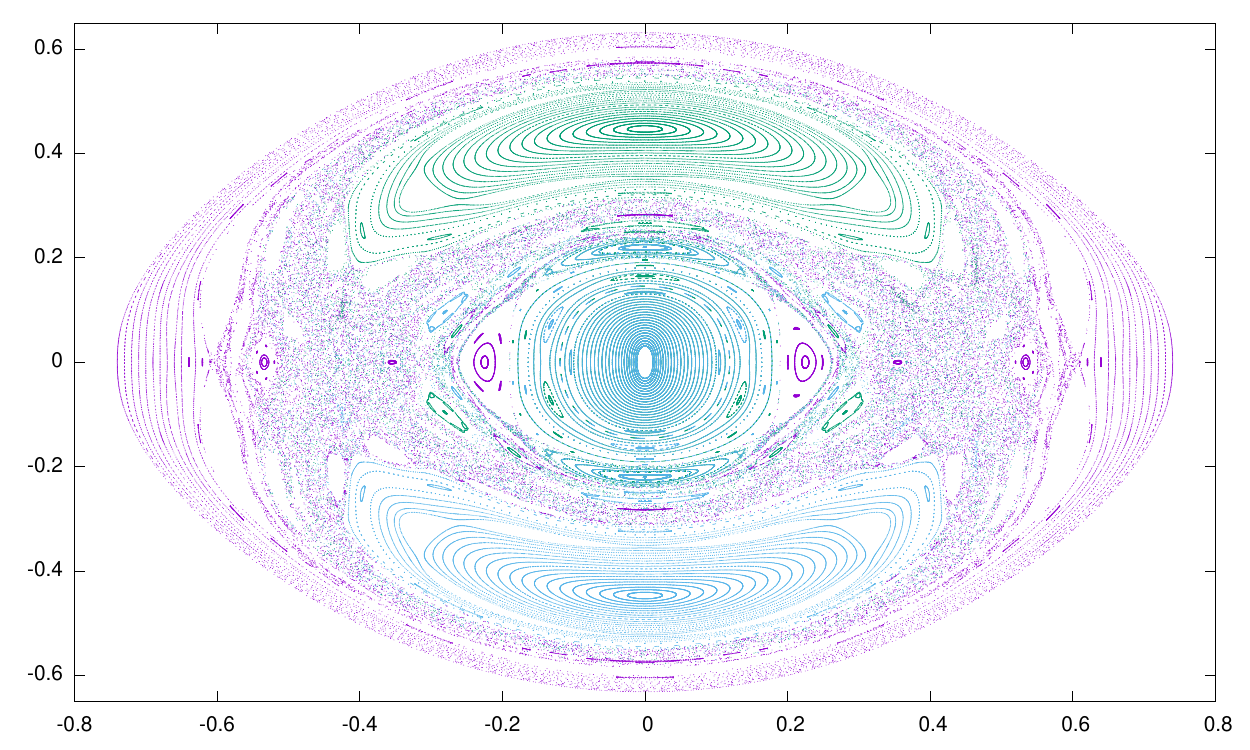}
\end{center}
\caption{Poincar\'e section $(x,p_x)$  with $h=0.2$,  $\mu=1$ starting at the integrable case $a=1$, $b=-1$. The integrable section is symmetrical as top Figure \ref{fig:4}   with $a=1$, $b=-2.5$ (top) there is a vertical bifurcation and  with $b=-5$ (bottom) the system is quite chaotic  but without escape trajectories as the case with $\mu=5$.}
\label{fig:5}      
\end{figure}

\begin{figure}
\begin{center}
\includegraphics[scale=0.9]{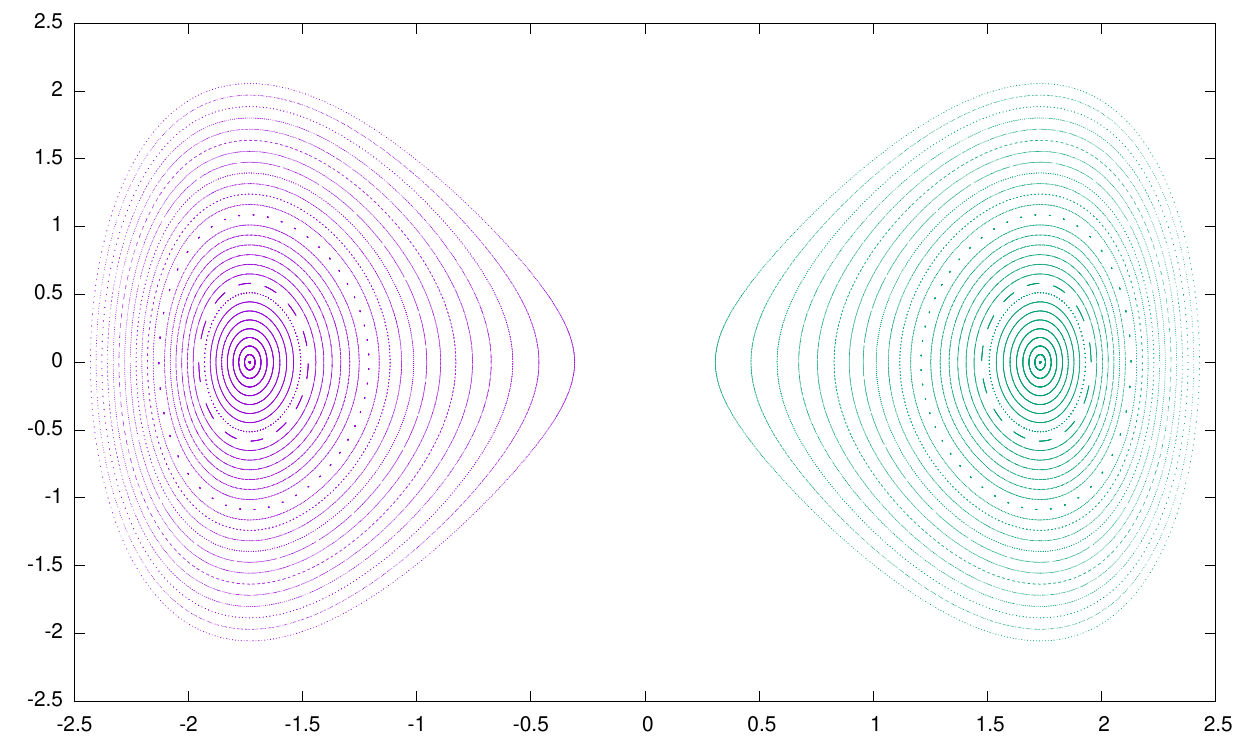}
\includegraphics[scale=0.9]{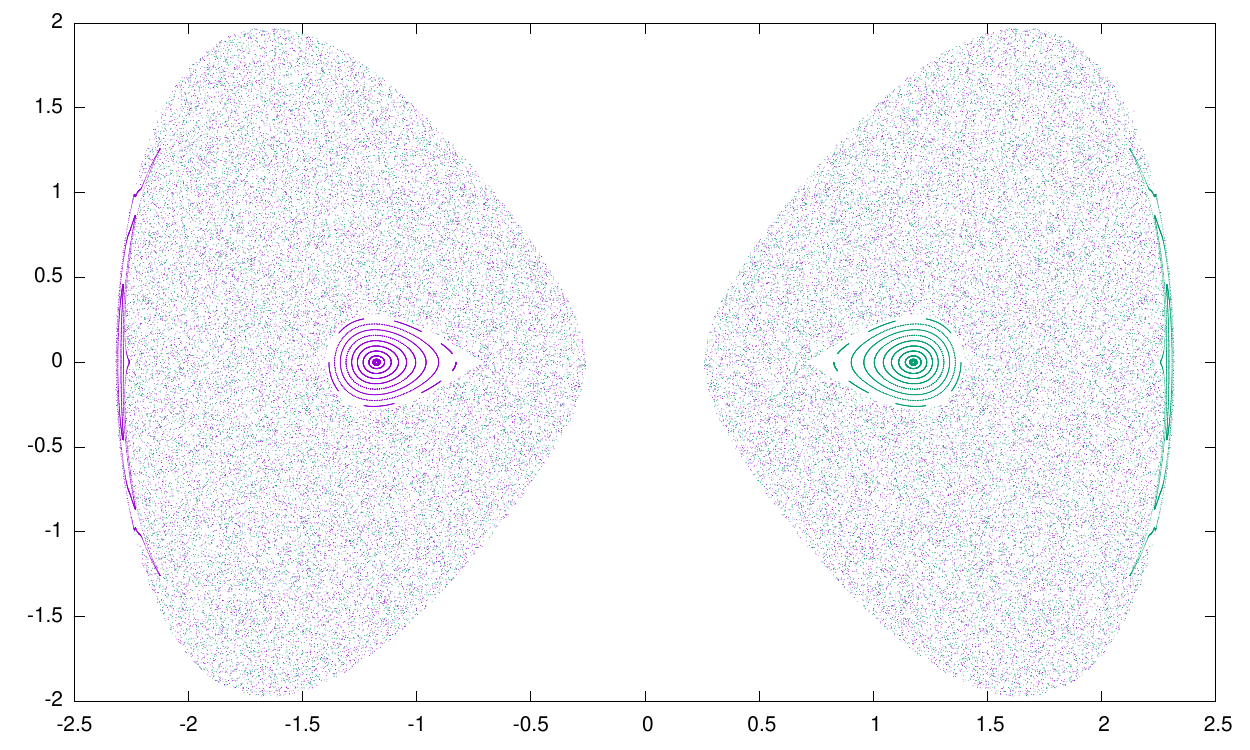}
\end{center}
\caption{Poincar\'e section $(x,p_x)$  with $h=-0.1$,  $\mu=-3$ starting at the integrable case $a=-1$, $b=1$ and the second  $a=-1.1$, $b=1$. Note how rapidly chaos appears with a small change in the $a$ parameter.}
\label{fig:6}      
\end{figure}

\section*{Acknowledgments}
We thank the anonymous referees for their  valuable comments
and suggestions helped us to improve this  paper.
We also thank  Thierry Combot for valuable discussions and providing his Maple code 
for testing homogeneous  potential and Maria Przybylska  who pointed out  us some papers on the integrability conditions for non-homogeneous potentials.
P. Acosta-Hum\'anez wishes to thank the Department of Mathematics, UAM-Iztapalapa, Mexico, where some of this work was carried out, and also to Universidad Sim\'on Bolivar (Barranquilla, Colombia) for its support in the final stage of this research.  M. Alvarez-Ram\'{\i}rez  is partially supported by the grant: Red de cuerpos acad\'emicos Ecuaciones Diferenciales. Proyecto sistemas  din\'amicos y estabilizaci\'on. PROMEP 2011-SEP, Mexico.

\bibliographystyle{siamplain}
\bibliography{refere}

\begin{thebibliography}{10}

\bibitem{am}
{\sc R.~Abraham and J.~E. Marsden}, {\em Foundations of mechanics}, Benjamin,
  Reading,  (1978).

\bibitem{abramowitz}
{\sc M.~Abramowitz and I.~A. Stegun}, {\em Handbook of mathematical functions:
  with formulas, graphs, and mathematical tables}, A Wiley-Interscience
  Publication, John Wiley \& Sons, Inc., New York; National Bureau of
  Standards, Washington, DC,  (1984).

\bibitem{acbl}
{\sc P.~Acosta-Hum\'anez and D.~Bl\'azquez-Sanz}, {\em Non-integrability of
  some hamiltonians with rational potential}, Discrete Contin. Dyn. Syst. Ser.
  B, 10 (2008), pp.~265--293,
  \href{http://dx.doi.org/10.3934/dcdsb.2008.10.265}
  {doi:10.3934/dcdsb.2008.10.265}.

\bibitem{acbook}
{\sc P.~B. Acosta-Hum\'anez}, {\em Galoisian approach to supersymmetric quantum
  mechanics. the integrability analysis of the schr\"odinger equation by means
  of differential galois theory}, VDM Verlag Dr. M\"uller, Berlin,  (2010).

\bibitem{almp}
{\sc P.~B. Acosta-Hum\'anez, J.~L\'azaro, J.~Morales-Ruiz, and C.~Pantazi},
  {\em On the integrability of polynomial vector fields in the plane by means
  of picard-vessiot theory}, Discrete Contin. Dyn. Syst., 35 (2015),
  pp.~1767--1800, \href{http://dx.doi.org/10.3934/dcds.2015.35.1767}
  {doi:10.3934/dcds.2015.35.1767}.

\bibitem{acmowe}
{\sc P.~B. Acosta-Hum\'anez, J.~Morales-Ruiz, and J.~A. Weil}, {\em Galoisian
  approach to integrability of schr\"odinger equation}, Rep. Math. Phys., 67
  (2011), pp.~305--374, \href{http://dx.doi.org/10.1016/S0034-4877(11)60019-0}
  {doi:10.1016/S0034-4877(11)60019-0}.

\bibitem{andrle}
{\sc P.~Andrle}, {\em A third integral of motion in a system with a potential
  of the fourth degree}, Phys. Lett. A, 17 (1966), pp.~169--175.

\bibitem{agk}
{\sc D.~Armbruster, J.~Guckenheimer, and S.~Kim}, {\em Chaotic dynamics in
  systems with square symmetry}, Phys. Lett. A, 140 (1989), pp.~416--420,
  \href{http://dx.doi.org/10.1016/0375-9601(89)90078-9}
  {doi:10.1016/0375-9601(89)90078-9}.

\bibitem{combot}
{\sc A.~Bostan, T.~Combot, and M.~E. Din}, {\em Computing necessary
  integrability conditions for planar parametrized homogeneous potentials}, In
  Proceedings of the 39th International Symposium on Symbolic and Algebraic
  Computation,  (2014), pp.~67--74,
  \href{http://dx.doi.org/10.1145/2608628.2608662}
  {doi:10.1145/2608628.2608662}.

\bibitem{contra}
{\sc G.~Contopoulos}, {\em Galactic dynamics}, Princeton University Press,
  (1988).

\bibitem{elmandouh}
{\sc A.~A. Elmandouh}, {\em On the dynamics of armbruster-guckenheimer-kim
  galactic potential in a rotating reference frame}, Astrophys Space Sci, 361
  (2016), p.~12, \href{http://dx.doi.org/10.1007/s10509-016-2770-8}
  {doi:10.1007/s10509-016-2770-8}.

\bibitem{hietarinta}
{\sc J.~Hietarinta}, {\em Direct methods for the search of the second
  invariant}, Phys. Rep., 147 (1987), pp.~87--154.

\bibitem{ka}
{\sc I.~Kaplansky}, {\em An introduction to differential algebra}, Hermann,
  (1957).

\bibitem{kol}
{\sc E.~Kolchin}, {\em Differential algebra and algebraic groups}, Academic
  Press, Pure and Applied Mathematics, New York - London, 54 (1973).

\bibitem{llibre}
{\sc J.~Llibre and L.~Roberto}, {\em Periodic orbits and non-integrability of
  armbruster-guckenheimer-kim potential}, Astrophys Space Sci, 343 (2013),
  pp.~69--74, \href{http://dx.doi.org/10.1007/s10509-012-1210-7}
  {doi:10.1007/s10509-012-1210-7}.

\bibitem{maria}
{\sc A.~Maciejewski and M.~Przybylska}, {\em Darboux points and integrability
  of hamiltonian systems with homogeneous polynomial potential}, J. Math.
  Phys., 46 (2005), p.~33 pp,
  \href{http://dx.doi.org/http://dx.doi.org/10.1063/1.1917311}
  {doi:http://dx.doi.org/10.1063/1.1917311}.

\bibitem{morales}
{\sc J.~Morales-Ruiz}, {\em Differential galois theory and non-integrability of
  hamiltonian systems}, Progress in Math., Birkhauser, Verlag, Basel, 178
  (1999).

\bibitem{morasurvey}
{\sc J.~Morales-Ruiz and J.~P. Ramis}, {\em Integrability of dynamical systems
  through differential galois theory: a practical guide}, Differential algebra,
  complex analysis and orthogonal polynomials, Contemp. Math., Amer. Math.
  Soc., Providence, RI, 509 (2010), pp.~143--220,
  \url{http://dx.doi.org/10.1090/conm/509}.

\bibitem{mora1}
{\sc J.~J. Morales-Ruiz and J.~P. Ramis}, {\em Galoisian obstructions to
  integrability of hamiltonian systems i}, Methods Appl. Anal., 8 (2001),
  pp.~33--96, \href{http://dx.doi.org/10.4310/MAA.2001.v8.n1.a3}
  {doi:10.4310/MAA.2001.v8.n1.a3}.

\bibitem{mora2}
{\sc J.~J. Morales-Ruiz and J.~P. Ramis}, {\em Galoisian obstructions to
  integrability of hamiltonian systems ii}, Methods Appl. Anal., 8 (2001),
  pp.~97--112, \href{http://dx.doi.org/10.4310/MAA.2001.v8.n1.a4}
  {doi:10.4310/MAA.2001.v8.n1.a4}.

\bibitem{Naka}
{\sc K.~Nakagawa}, {\em Direct construction of polynomial first integrals for
  hamiltonian systems with a two-dimensional homogeneous polynomial potential},
  Dep. of Astronomical Science, The Graduate University for Advanced Study and
  the National Astronomical Observatory of Japan, Ph.D. Thesis,  (2002).

\bibitem{poschl}
{\sc G.~P\"oschl and E.~Teller}, {\em Bemerkungen zur quantenmechanik des
  anharmonischen oszillators}, Zeitschrift f\"ur Physik., 83 (1933),
  pp.~143--151.

\bibitem{sg}
{\sc F.~Simonelli and J.~P. Gollub}, {\em Surface wave mode interactions:
  effects of symmetry and degeneracy}, J. Fluid Mech., 199 (1989),
  pp.~471--494, \url{http://dx.doi.org/10.1017/S0022112089000443}.

\bibitem{T}
{\sc E.~Titchmarsh}, {\em Eigenfunction expansions associated with second-order
  differential equations, part i}, (2nd edition). Oxford University Press,
  Oxford,  (1962).

\bibitem{vasi}
{\sc M.~van~der Put and M.~Singer}, {\em Galois theory of linear differential
  equations}, Grundlehren der Mathematischen Wissenschaften, Springer Verlag,
  Berlin, 328 (2003).

\bibitem{wigg}
{\sc S.~Wiggins}, {\em Introduction to applied nonlinear dynamical systems and
  chaos}, New York: Springer-Verlag,  (1990).

\bibitem{yoshida}
{\sc H.~Yoshida}, {\em Nonintegrability of the truncated toda lattice
  hamiltonian at any order}, Comm. Math. Phys., 116 (1988), pp.~529--538.

\end{thebibliography}

\end{document}